\documentclass[a4paper,12pt,reqno]{amsart}
\usepackage{amsfonts}
\usepackage{amsmath}
\usepackage{amssymb}
\usepackage[a4paper]{geometry}
\usepackage{mathrsfs}
\usepackage{csquotes}
\usepackage{enumitem}

\usepackage[colorlinks]{hyperref}
\renewcommand\eqref[1]{(\ref{#1})} %Need with hyperref
\numberwithin{equation}{section}
\theoremstyle{plain}
\newtheorem{thm}{Theorem}[section]

\newtheorem{lem}[thm]{Lemma}
 
\theoremstyle{definition}
\newtheorem{defn}[thm]{Definition}

\newcommand{\Rn}{\mathbb R^{n}}
\newcommand{\G}{\mathbb G}

\def\R{\mathcal R}

\def\R{\mathcal R}

\def\e[#1]{{\textrm{e}}^{#1}}

\def\G{{\mathbb G}}

\begin{document}

   \title[Very weak solutions to hypoelliptic wave equations]
 {Very weak solutions to hypoelliptic wave equations}

%    \title[Very weak solutions to wave equations on graded groups]{Very weak solutions to wave equations on graded groups}

\author[M. Ruzhansky]{Michael Ruzhansky}
\address{
  Michael Ruzhansky:
  \endgraf
  Department of Mathematics
  \endgraf
  Ghent University, Belgium
  \endgraf
  and
  \endgraf
  School of Mathematical Sciences
    \endgraf
    Queen Mary University of London
  \endgraf
  United Kingdom
  \endgraf
  {\it E-mail address} {\rm ruzhansky@gmail.com}
  }

\author[N. Yessirkegenov]{Nurgissa Yessirkegenov}
\address{
  Nurgissa Yessirkegenov:
  \endgraf
  Institute of Mathematics and Mathematical Modelling
  \endgraf
  125 Pushkin str.
  \endgraf
  050010 Almaty
  \endgraf
  Kazakhstan
  \endgraf
  and
  \endgraf
  Department of Mathematics
  \endgraf
  Imperial College London
  \endgraf
  180 Queen's Gate, London SW7 2AZ
  \endgraf
  United Kingdom
  \endgraf
  {\it E-mail address} {\rm n.yessirkegenov15@imperial.ac.uk}
  }

\thanks{The first author was supported in parts by the FWO Odysseus Project, by the EPSRC grant EP/R003025/1 and by the Leverhulme
Grant RPG-2017-151. The second author was supported by the MESRK grant AP05133271. No new data was collected
 or
generated during the course of research.}


     \begin{abstract} In this paper we study the Cauchy problem for the wave equations for hypoelliptic homogeneous
left-invariant operators on graded Lie groups when the time-dependent non-negative propagation speed
is regular, H\"{o}lder, and distributional. For H\"{o}lder coefficients we derive the well-posedness in the spaces of ultradistributions associated to Rockland operators on graded groups. In the case when the propagation speed is a distribution, we employ the notion of \enquote{very weak solutions} to the Cauchy problem, that was already successfully used in similar contexts in \cite{GR15} and \cite{RT17_phys_let}. We show that the Cauchy problem for the
wave equation with the distributional coefficient has a unique \enquote{very weak solution}
in an appropriate sense, which coincides with classical or distributional solutions
when the latter exist. Examples include the time dependent wave equation for the sub-Laplacian on the Heisenberg group or on
general stratified Lie groups, or $p$-evolution equations for higher order operators on $\Rn$ or on groups, the results already being new in all these cases.
     \end{abstract}
     \maketitle

       \tableofcontents

\section{Introduction}
\label{SEC:intro}
In this paper we are interested in the well-posedness of the following Cauchy problem for general positive hypoelliptic (Rockland operators of homogeneous degree $\nu$) left-invariant differential operators $\R$ on general graded Lie group $\G$ with the non-negative propagation speed $a=a(t)$ and with the source term $f=f(t)\in L^{2}(\G)$:
\begin{equation}\label{intro_eq1}
\left\{
                \begin{array}{ll}
                  \partial_{t}^{2}u(t)+a(t)\R u(t)=f(t), \;t\in [0,T],\\
                  u(0)=u_{0}\in L^{2}(\G),\\
                  \partial_{t}u(0)=u_{1}\in L^{2}(\G).
                \end{array}
              \right.
\end{equation}
When $\G=(\Rn,+)$ and $\R=-\triangle$ is the positive Laplacian, that is, the equation in \eqref{intro_eq1} is the usual wave equation, the well-posedness of the Cauchy problem \eqref{intro_eq1} for H\"{o}lder functions $a=a(t)$ goes back to Colombini, de Giorgi and Spagnolo \cite{CDS79}. In \cite{CJS87} and \cite{CS82}, it was also shown that when $\G=(\mathbb{R},+)$ and $\R=-\frac{d^{2}}{dx^{2}}$, the Cauchy problem \eqref{intro_eq1} does not have to be well-posed in $C^{\infty}$ if $a\in C^{\infty}$ is not strictly positive or if it is in the H\"{o}lder class $a\in C^{\alpha}$ for $0<\alpha<1$.

We note that following the seminal paper by Rothschild and Stein \cite{RS76}, such Rockland operators can be considered as model `approximations' of general hypoelliptic partial differential operators on manifolds.

Before discussing the obtained results on graded groups, let us briefly recall some necessary facts. The Sobolev space $H_{\R}^{s}(\G)$ for any $s\in\mathbb{R}$ is the subspace of $S^{\prime}(\G)$ obtained as the completion of the Schwartz space $S(\G)$ with respect to the Sobolev norm
\begin{equation}\label{Sob_norm}
\|f\|_{H_{\R}^{s}(\G)}:=\|(I+\R)^{\frac{s}{\nu}}f\|_{L^{2}(\G)}.
\end{equation}
In the case of stratified Lie groups such spaces and their properties have been extensively analysed by Folland in \cite{Fol75} and on general graded Lie groups they have been investigated in \cite{FR16} and \cite{FR17}. Recall that these spaces do not depend on a particular choice of the Rockland operator $\R$ used in the definition \eqref{Sob_norm}, see \cite[Theorem 4.4.20]{FR16}.

A brief review of the necessary notions related to graded Lie groups will be given in Section \ref{SEC:prelim}. We will also use $\R$-Gevrey (Roumieu) $\mathcal{G}_{\R}^{s}(\G)$ and $\R$-Gevrey (Beurling) type spaces $\mathcal{G}_{\R}^{(s)}(\G)$ for $s\geq1$, which are defined by
\begin{equation}\label{Gevrey_rou}
\mathcal{G}_{\R}^{s}(\G):=\{f\in C^{\infty}(\G)|\exists A>0:\|e^{A\R^{\frac{1}{2s}}}f\|_{L^{2}(\G)}<\infty\}
\end{equation}
and
\begin{equation}\label{Gevrey_beu}
\mathcal{G}_{\R}^{(s)}(\G):=\{f\in C^{\infty}(\G)|\forall A>0:\|e^{A\R^{\frac{1}{2s}}}f\|_{L^{2}(\G)}<\infty\},
\end{equation}
respectively.

Recently, in \cite{RT17a} (see also \cite[Theorem 3.1.1]{Tar18}), the following well-posedness result in the case of the homogeneous Cauchy problem \eqref{intro_eq1} (i.e. when $f\equiv0$) was obtained:
\begin{thm}[{\cite[Theorem 1.1]{RT17a} or \cite[Theorem 3.1.1]{Tar18}}]
\label{Ruzh_Tar_thm}
Let $\G$ be a graded Lie group and let $\R$ be a positive Rockland operator
of homogeneous degree $\nu$. Let $T>0$. Then we have
\begin{enumerate}[label=(\roman*)]
\item Let $a\in{\rm Lip}([0,T])$ with $a(t)\geq a_{0}>0$. Given $s\in\mathbb{R}$, if the initial Cauchy data $(u_{0},u_{1})$ are in $H_{\R}^{s+\frac{\nu}{2}}(\G)\times H_{\R}^{s}(\G)$, then there exists the unique solution of the homogeneous Cauchy problem \eqref{intro_eq1} (when $f\equiv 0$) in the space $C([0,T],H_{\R}^{s+\frac{\nu}{2}}(\G))\cap C^{1}([0,T],H_{\R}^{s}(\G))$, satisfying the following inequality for all values of $t\in [0,T]$:
    \begin{equation}\label{est_tar1}
    \|u(t,\cdot)\|^{2}_{ H^{s+\frac{\nu}{2}}_{\R}(\G)}+\|\partial_{t} u(t,\cdot)\|^{2}_{ H^{s}_{\R}(\G)}
\leq C(\|u_{0}\|^{2}_{H^{s+\frac{\nu}{2}}_{\R}(\G)}+\|u_{1}\|^{2}_{H^{s}_{\R}(\G)});
    \end{equation}
\item Let $a\in C^{\alpha}([0,T])$ with $0<\alpha<1$ and $a(t)\geq a_{0}>0$. If the initial Cauchy data  $(u_{0},u_{1})$ are in $\mathcal{G}_{\R}^{s}(\G)\times\mathcal{G}_{\R}^{s}(\G)$, then there exists the
unique solution of the homogeneous Cauchy problem \eqref{intro_eq1} (when $f\equiv 0$) in $C^{2}([0,T],\mathcal{G}_{\R}^{s}(\G))$, provided that
$$1\leq s<1+\frac{\alpha}{1-\alpha};$$
\item Let $a\in C^{\ell}([0,T])$ with $\ell\geq2$ and $a(t)\geq0$. If the initial Cauchy data  $(u_{0},u_{1})$ are in $\mathcal{G}_{\R}^{s}(\G)\times\mathcal{G}_{\R}^{s}(\G)$, then there exists the unique solution of the homogeneous Cauchy problem \eqref{intro_eq1} (when $f\equiv 0$) in $C^{2}([0,T],\mathcal{G}_{\R}^{s}(\G))$, provided that
$$1\leq s<1+\frac{\ell}{2};$$
\item Let $a\in C^{\alpha}([0,T])$ with $0<\alpha<2$ and $a(t)\geq 0$. If the initial Cauchy data  $(u_{0},u_{1})$ are in $\mathcal{G}_{\R}^{s}(\G)\times\mathcal{G}_{\R}^{s}(\G)$, then there exists the unique solution of the homogeneous Cauchy problem \eqref{intro_eq1} (when $f\equiv 0$) in $C^{2}([0,T],\mathcal{G}_{\R}^{s}(\G))$, provided that
$$1\leq s<1+\frac{\alpha}{2}.$$
\end{enumerate}
\end{thm}

Let $H_{s}^{-\infty}$ and $H_{(s)}^{-\infty}$ be the spaces of linear continuous functionals on $\mathcal{G}_{\R}^{s}$ and $\mathcal{G}_{\R}^{(s)}$, respectively.

In particular, in this paper we show the inhomogeneous case of Theorem \ref{Ruzh_Tar_thm} and the case when the initial Cauchy data $(u_{0},u_{1})$ can be also from the space $H_{(s)}^{-\infty}$:
\begin{thm}\label{Ruzh_Tar_thm_inhom}
Let $\G$ be a graded Lie group and let $\R$ be a positive Rockland operator
of homogeneous degree $\nu$. Let $T>0$. Then we have
\begin{enumerate}[label=(\roman*)]
\item Let $a\in{\rm Lip}([0,T])$ with $a(t)\geq a_{0}>0$. Given $s\in\mathbb{R}$, if $f\in C([0, T ], H^{s}_{\R}(\G))$ and the initial Cauchy data $(u_{0},u_{1})$ are in $H_{\R}^{s+\frac{\nu}{2}}(\G)\times H_{\R}^{s}(\G)$, then there exists a unique solution of \eqref{intro_eq1} in the space $C([0,T],H_{\R}^{s+\frac{\nu}{2}}(\G))\cap C^{1}([0,T],H_{\R}^{s}(\G))$, satisfying the following inequality for all values of $t\in [0,T]$:
    \begin{multline}\label{est_tar1_inhom}
    \|u(t,\cdot)\|^{2}_{ H^{s+\frac{\nu}{2}}_{\R}(\G)}+\|\partial_{t} u(t,\cdot)\|^{2}_{ H^{s}_{\R}(\G)}\\
\leq C(\|u_{0}\|^{2}_{H^{s+\frac{\nu}{2}}_{\R}(\G)}+\|u_{1}\|^{2}_{H^{s}_{\R}(\G)}+\|f\|^{2}_{C([0,T],H^{s}_{\R}(\G))});
    \end{multline}
\item Let $a\in C^{\alpha}([0,T])$ with $0<\alpha<1$ and $a(t)\geq a_{0}>0$. Then for initial data and for source term
\begin{enumerate}[label=(\alph*)]
\item $u_{0},u_{1}\in \mathcal{G}_{\R}^{s}(\G)$, $f\in C([0, T ], \mathcal{G}_{\R}^{s}(\G)$,
\item $u_{0},u_{1}\in H_{(s)}^{-\infty}$, $f\in C([0, T ], H_{(s)}^{-\infty})$,
\end{enumerate}
the Cauchy problem \eqref{intro_eq1} has the unique solutions
\begin{enumerate}[label=(\alph*)]
\item $u\in C^{2}([0,T];\mathcal{G}_{\R}^{s}(\G))$,
\item $u\in C^{2}([0,T];H_{(s)}^{-\infty})$,
\end{enumerate}
respectively, provided that
$$1\leq s<1+\frac{\alpha}{1-\alpha};$$
\item Let $a\in C^{\ell}([0,T])$ with $\ell\geq2$ and $a(t)\geq0$. Then for initial data and for source term
\begin{enumerate}[label=(\alph*)]
\item $u_{0},u_{1}\in \mathcal{G}_{\R}^{s}(\G)$, $f\in C([0, T ], \mathcal{G}_{\R}^{s}(\G))$,
\item $u_{0},u_{1}\in H_{(s)}^{-\infty}$, $f\in C([0, T ], H_{(s)}^{-\infty})$,
\end{enumerate}
the Cauchy problem \eqref{intro_eq1} has the unique solutions
\begin{enumerate}[label=(\alph*)]
\item $u\in C^{2}([0,T];\mathcal{G}_{\R}^{s}(\G))$,
\item $u\in C^{2}([0,T];H_{(s)}^{-\infty})$,
\end{enumerate}
respectively, provided that
$$1\leq s<1+\frac{\ell}{2};$$
\item Let $a\in C^{\alpha}([0,T])$ with $0<\alpha<2$ and $a(t)\geq 0$. Then for initial data and for source term
\begin{enumerate}[label=(\alph*)]
\item $u_{0},u_{1}\in \mathcal{G}_{\R}^{s}(\G)$, $f\in C([0, T ], \mathcal{G}_{\R}^{s}(\G))$,
\item $u_{0},u_{1}\in H_{(s)}^{-\infty}$, $f\in C([0, T ], H_{(s)}^{-\infty})$,
\end{enumerate}
the Cauchy problem \eqref{intro_eq1} has the unique solutions
\begin{enumerate}[label=(\alph*)]
\item $u\in C^{2}([0,T];\mathcal{G}_{\R}^{s}(\G))$,
\item $u\in C^{2}([0,T];H_{(s)}^{-\infty})$,
\end{enumerate}
respectively, provided that
$$1\leq s<1+\frac{\alpha}{2}.$$
\end{enumerate}
\end{thm}
Since we are also interested in the case when the time-dependent propagation speed $a$ is less regular than H\"{o}lder, let us recall some results in this direction. In \cite{GR15}, the authors introduced the notion of \enquote{very weak solutions} for a wave-type second order invariant partial differential operator in $\Rn$, and proved their existence, uniqueness and consistency with classical or distributional solutions should the latter exist. For similar results in $\Rn$, we also refer to \cite{RT17_phys_let} for the Landau Hamiltonian, and to \cite{RT17_arch} for operators with a discrete non-negative spectrum. Thus, the second aim of this paper is to carry out similar investigations for general hypoelliptic operators, namely for positive Rockland operators \eqref{intro_eq1}, whose spectrum is absolutely continuous. To give some examples, this setting includes:
\begin{itemize}
\item for $\mathbb{G}=\mathbb{R}^{n}$, $\mathcal{R}$ may be any positive homogeneous elliptic differential operator with constant coefficients. For example, we can take
    $$\mathcal{R}=(-\triangle)^{m} \;\;{\rm or}\;\;\mathcal{R}=(-1)^{m}\sum_{j=1}^{n}a_{j}\left(\frac{\partial}{\partial x_{j}}\right)^{2m},$$
    where $a_{j}>0$ and $m\in\mathbb{N}$;
\item for the Heisenberg group $\mathbb{G}=\mathbb{H}_{n}$, we can take
$$\mathcal{R}=(-\mathcal{L})^{m}\;{\rm or}\;\mathcal{R}=(-1)^{m}\sum_{j=1}^{n}(a_{j}X_{j}^{2m}+b_{j}Y_{j}^{2m}),$$
where $a_{j}, b_{j}>0, m\in \mathbb{N}$, and $\mathcal{L}=\sum_{j=1}^{n}(X_{j}^{2}+Y_{j}^{2})$ is the sub-Laplacian, and $$X_{j}:=\partial_{x_{j}}-\frac{y_{j}}{2}\partial_{t}, \;Y_{j}:=\partial_{y_{j}}+\frac{x_{j}}{2}\partial_{t};$$
\item for any stratified Lie group (or homogeneous Carnot group) with vectors $X_{1},\ldots, X_{k}$ spanning the first stratum, we can take
    $$\mathcal{R}=(-1)^{m}\sum_{j=1}^{k}a_{j}X_{j}^{2m}, \;a_{j}>0,$$
    so that in particular, for $m=1$, $\mathcal{R}$ is a positive sub-Laplacian;
    \item for any graded Lie group $\mathbb{G}\sim\mathbb{R}^{n}$ with dilation weights $\nu_{1},\ldots,\nu_{n}$ let us fix the basis $X_{1},\ldots,X_{n}$ of the Lie algebra $\mathfrak{g}$ of $\mathbb{G}$ satisfying
        $$D_{r}X_{j}=r^{\nu_{j}}X_{j}, \;j=1,\ldots,n, \;r>0,$$
        where $D_{r}$ denote dilations on the Lie algebra. If $\nu_{0}$ is any common multiple of $\nu_{1},\ldots,\nu_{n}$, the operator
        $$\mathcal{R}=\sum_{j=1}^{n}(-1)^{\frac{\nu_{0}}{\nu_{j}}}a_{j}X_{j}^{2\frac{\nu_{0}}{\nu_{j}}},
        a_{j}>0$$
        is a Rockland operator of homogeneous degree $2\nu_{0}$.
\end{itemize}
Now we shall describe the notion of very weak solutions and formulate the corresponding results for distributions
$a\in \mathcal{D}^{\prime}([0,T])$ and $f\in \mathcal{D}^{\prime}([0,T]) \bar{\bigotimes} H_{\R}^{-\infty}$. First, we regularise the distributional coefficient $a$ and the source term $f$ by the convolution with a suitable mollifier $\psi$ obtaining families
of smooth functions $(a_{\varepsilon})_{\varepsilon}$ and $(f_{\varepsilon})_{\varepsilon}$ as follows
\begin{equation}\label{ae_fe}
a_{\varepsilon}=a\ast\psi_{\omega(\varepsilon)},\;\;f_{\varepsilon}=f(\cdot)\ast\psi_{\omega(\varepsilon)},
\end{equation}
with $$\psi_{\omega(\varepsilon)}(t)=(\omega(\varepsilon))^{-1}\psi(t/\omega(\varepsilon)),$$
where $\omega(\varepsilon)>0$ (which we will choose later) is such that $\omega(\varepsilon)\rightarrow0$ as $\varepsilon\rightarrow0$, and $\psi$ is a Friedrichs-mollifier, that is,
$$\psi\in C_{0}^{\infty}(\mathbb{R}),\;\;\psi\geq 0\;\;\text{and}\;\;\int \psi=1.$$

Let us give the following definition:
\begin{defn}\label{moderate_def}
\begin{enumerate}[label=(\roman*)]
\item A net of functions $(f_{\varepsilon})_{\varepsilon}\in C^{\infty}(\mathbb{R})^{(0,1]}$ is said to be $C^{\infty}$-
moderate if for all $K\Subset \mathbb{R}$ and for all $\alpha\in \mathbb{N}_{0}$ there exist $N\in \mathbb{N}_{0}$ and $c>0$ such that
$$\sup_{t\in K}|\partial^{\alpha}f_{\varepsilon}(t)|\leq c\varepsilon^{-N-\alpha},$$
for all $\varepsilon \in (0,1]$, where $K\Subset \mathbb{R}$ means that $K$ is a compact set in $\mathbb{R}$.

\item A net of functions $(u_{\varepsilon})_{\varepsilon}\in C^{\infty}([0,T];H_{\R}^{s})^{(0,1]}$ is $C^{\infty}([0,T];H_{\R}^{s})$-moderate if there exist $N\in \mathbb{N}_{0}$ and $c_{k}>0$ for all $k\in\mathbb{N}_{0}$ such that
$$\|\partial_{t}^{k}u_{\varepsilon}(t,\cdot)\|_{H^{s}_{\R}}\leq c_{k}\varepsilon^{-N-k},$$
for all $t\in [0,T]$ and $\varepsilon \in (0,1]$.
\item We say that a net of functions $(u_{\varepsilon})_{\varepsilon}\in C^{\infty}([0,T];H_{(s)}^{-\infty})^{(0,1]}$ is $C^{\infty}([0,T];H_{(s)}^{-\infty})$-moderate if there exists $\eta>0$ and, for all $p\in \mathbb{N}_{0}$ there exists $c_{p}>0$ and $N_{p}>0$ such that
    $$\|e^{-\eta \R^{\frac{1}{2s}}}\partial_{t}^{p}u_{\varepsilon}(t,\cdot)\|_{L^{2}(\G)}\leq c_{p}\varepsilon ^{-N_{p}-p},$$
    for all $t\in [0,T]$ and $\varepsilon\in(0,1]$.
\end{enumerate}
\end{defn}
It turns out that if e.g. $\omega(\varepsilon)=\varepsilon$, then the net $(a_{\varepsilon})_{\varepsilon}$ in \eqref{ae_fe} is $C^{\infty}$-moderate. Note that the conditions of moderateness are natural in the sense that regularisations of distributions are moderate, namely by the structure theorems for distributions one can regard
\begin{equation}\label{com_sup_mod}
\text{compactly supported distributions}\;\;\mathcal{E}^{\prime}(\mathbb{R})\subset\{C^{\infty}\text{-moderate families}\}.
\end{equation}
Thus, by \eqref{com_sup_mod} we see that while a solution to the Cauchy problems may not exist in the space of
distributions $\mathcal{E}^{\prime}(\mathbb{R})$, it may still exist (in a certain appropriate sense) in the space on the right hand side of \eqref{com_sup_mod}. The moderateness assumption allows us to recapture the solution as in \eqref{est_tar1_inhom} when it exists. However, we note that regularisation with standard Friedrichs mollifiers is not always sufficient, hence the introduction of a family $\omega(\varepsilon)$ in the above regularisations.

Now let us introduce a notion of a \enquote{very weak solution} for the Cauchy problem
\begin{equation}\label{weak_prob_state}
\left\{
                \begin{array}{ll}
                  \partial_{t}^{2}u(t)+a(t)\R u(t)=f(t), \;t\in [0,T],\\
                  u(0)=u_{0}\in L^{2}(\G),\\
                  \partial_{t}u(0)=u_{1}\in L^{2}(\G).
                \end{array}
              \right.
\end{equation}
\begin{defn}\label{veryweak_def} Let $s$ be a real number.
\begin{enumerate}[label=(\roman*)] \item We say that the net $(u_{\varepsilon})_{\varepsilon}\subset C^{\infty}([0,T];H_{\R}^{s})$ is a very weak solution of $H^{s}$-type of the Cauchy problem \eqref{weak_prob_state} if there exist

$C^{\infty}$-moderate regularisation $a_{\varepsilon}$ of the coefficient $a$,

$C^{\infty}([0,T];H_{\R}^{s})$-moderate regularisation $f_{\varepsilon}(t)$ of $f(t)$, such that $(u_{\varepsilon})_{\varepsilon}$ solves the following regularised problem
\begin{equation}\label{epsil_weak_prob_state}
\left\{
                \begin{array}{ll}
                  \partial_{t}^{2}u_{\varepsilon}(t)+a_{\varepsilon}(t)\R u_{\varepsilon}(t)=f_{\varepsilon}(t), \;t\in [0,T],\\
                  u_{\varepsilon}(0)=u_{0}\in L^{2}(\G),\\
                  \partial_{t}u_{\varepsilon}(0)=u_{1}\in L^{2}(\G),
                \end{array}
              \right.
\end{equation}
for all $\varepsilon \in (0,1]$, and is $C^{\infty}([0,T];H_{\R}^{s})$-moderate.
\item The net $(u_{\varepsilon})_{\varepsilon}\subset C^{\infty}([0,T];H^{-\infty}_{(s)})$ is a very weak solution of $H_{(s)}^{-\infty}$-type of the Cauchy problem \eqref{weak_prob_state} if there exist

$C^{\infty}$-moderate regularisation $a_{\varepsilon}$ of the coefficient $a$,

$C^{\infty}([0,T];H_{(s)}^{-\infty})$-moderate regularisation $f_{\varepsilon}(t)$ of $f(t)$, such that $(u_{\varepsilon})_{\varepsilon}$ solves the regularised problem \eqref{epsil_weak_prob_state} for all $\varepsilon\in (0,1]$, and is $C^{\infty}([0,T];H_{(s)}^{-\infty})$-moderate.
\end{enumerate}
\end{defn}
Note that by Theorem \ref{Ruzh_Tar_thm_inhom} (i), we know that the Cauchy problem \eqref{epsil_weak_prob_state} has
a unique solution satisfying estimate \eqref{est_tar1_inhom}.

As usual, $a$ is a nonnegative distribution means that there exists a constant $a_{0}>0$ such
that $a\geq a_{0}>0$, while $a\geq a_{0}$ means that $a-a_{0}\geq 0$, i.e. $\langle a-a_{0}, \psi\rangle\geq 0$ for all nonnegative $\psi\in C_{0}^{\infty}(\mathbb{R})$. It can be remarked that it follows then that $a$ is actually a positive measure, although we will not need to make use of this fact in our analysis.

Thus, let us formulate the result of the paper on the existence of very weak solutions of the Cauchy problem \eqref{weak_prob_state}.
\begin{thm} \label{exist_weak} (Existence) Let $\G$ be a graded Lie group and let $\R$ be a positive Rockland operator
of homogeneous degree $\nu$. Let $T>0$ and $s\in \mathbb{R}$.
\begin{enumerate}[label=(\roman*)] \item
Let $a=a(t)$ be a positive distribution with compact support included in $[0,T]$, such that $a\geq a_{0}$ for some constant $a_{0}>0$. Let $u_{0}, u_{1}\in H_{\R}^{s}$ and $f\in \mathcal{D}^{\prime}([0,T])\bar{\bigotimes}H_{\R}^{s}$. Then the Cauchy problem \eqref{weak_prob_state} has a very weak solution of $H^{s}$-type.
\item Let $a=a(t)$ be a nonnegative distribution with compact support included in $[0,T]$, such that $a\geq 0$. Let $u_{0}, u_{1}\in H_{(s)}^{-\infty}$ and $f\in \mathcal{D}^{\prime}([0,T])\bar{\bigotimes}H_{(s)}^{-\infty}$. Then the Cauchy problem \eqref{weak_prob_state} has a very weak solution of $H_{(s)}^{-\infty}$-type.
\end{enumerate}
\end{thm}
Now we show that the very weak solution of the Cauchy problem \eqref{weak_prob_state} is unique in an appropriate
sense. For the formulation of the uniqueness statement, we will use the language of Colombeau algebras.
\begin{defn}\label{uniq_defn} The net $(u_{\varepsilon})_{\varepsilon}$ is $C^{\infty}$-negligible if for all $K\Subset \mathbb{R}$, for all $\alpha\in \mathbb{N}$ and for all $\ell\in \mathbb{N}$ there exists a positive constant $c$ such that
$$\sup_{t\in K}|\partial^{\alpha}u_{\varepsilon}(t)|\leq c\varepsilon^{\ell},$$
for all $\varepsilon \in (0,1]$.
\end{defn}
Actually, in this paper, it is sufficient to take $K=[0,T]$, since the time-dependent distributions can be taken supported in the interval $[0,T]$.

Let us now introduce the Colombeau algebra in the following quotient form:
$$\mathcal{G}(\mathbb{R})=\frac{C^{\infty}-\;\text{moderate\;nets}}{C^{\infty}-\;\text{negligible\;nets}}.$$
We refer to e.g. \cite{Obe92} for the general analysis of $\mathcal{G}(\mathbb{R})$ .
\begin{thm}\label{uniq_exist} (Uniqueness) Let $\G$ be a graded Lie group and let $\R$ be a positive Rockland operator
of homogeneous degree $\nu$. Let $T>0$.
\begin{enumerate}[label=(\roman*)] \item Let $a=a(t)$ be a positive distribution with compact support included in $[0,T]$, such
that $a(t)\geq a_{0}$ for some constant $a_{0}>0$. Let $(u_{0}, u_{1})\in H_{\R}^{s+\frac{\nu}{2}}\times H_{\R}^{s}$ and $f\in \mathcal{G}([0,T];H_{\R}^{s})$ for some $s\in \mathbb{R}$. Then there exists an embedding of the coefficient $a(t)$ into $\mathcal{G}([0,T])$, such that the Cauchy problem \eqref{weak_prob_state} has a unique solution $u\in \mathcal{G}([0,T]; H^{s}_{\R})$.

\item Let $a=a(t)\geq 0$ be a nonnegative distribution with compact support included in $[0,T]$. Let $(u_{0}, u_{1})\in H_{(s)}^{-\infty}$ and $f\in \mathcal{G}([0,T];H_{(s)}^{-\infty})$ for some $s\in \mathbb{R}$. Then there exists an embedding of the coefficient $a(t)$ into $\mathcal{G}([0,T])$, such that the Cauchy problem \eqref{weak_prob_state} has a unique solution $u\in \mathcal{G}([0,T]; H_{(s)}^{-\infty})$.
\end{enumerate}
\end{thm}
Now we give the consistency result, which means that very weak solutions recapture the classical solutions in the case the latter exist. For instance, we can compare the solution given by Theorem \ref{Ruzh_Tar_thm_inhom} (i) and Part (b) of Theorem \ref{Ruzh_Tar_thm_inhom} (iii) with the very weak solutions in Theorem \ref{exist_weak} under assumptions when Theorem \ref{Ruzh_Tar_thm_inhom} (i) and Part (b) of Theorem \ref{Ruzh_Tar_thm_inhom} (iii) hold.

Denote by $L^{\infty}_{1}([0,T])$ the space of bounded functions on $[0,T]$ with the derivative also in $L^{\infty}$.
\begin{thm} \label{consist_weak} (Consistency-1) Let $\G$ be a graded Lie group and let $\R$ be a positive Rockland operator
of homogeneous degree $\nu$. Let $T>0$.

\begin{enumerate}[label=(\roman*)] \item Let $a\in L^{\infty}_{1}([0,T])$ with $a(t)\geq a_{0}>0$. Let $s\in\mathbb{R}$, $(u_{0}, u_{1})\in H_{\R}^{s+\frac{\nu}{2}}\times H_{\R}^{s}$ and $f\in C([0,T];H_{\R}^{s})$. Let $u$ be a very weak solution of $H^{s}$-type of \eqref{weak_prob_state}. Then for any regularising families $a_{\varepsilon}$ and $f_{\varepsilon}$ in Definition \ref{veryweak_def}, any representative $(u_{\varepsilon})_{\varepsilon}$ of $u$ converges in $C([0,T]; H^{s+\frac{\nu}{2}}_{\R})\cap C^{1}([0,T];H^{s}_{\R})$ as $\varepsilon\rightarrow0$ to the unique classical
solution in $C([0,T]; H^{s+\frac{\nu}{2}}_{\R})\cap C^{1}([0,T];H^{s}_{\R})$ of the Cauchy problem \eqref{weak_prob_state} given by Theorem \ref{Ruzh_Tar_thm_inhom} (i).

\item Let $a\in C^{\ell}([0,T])$ with $\ell\geq 2$ be such that $a(t)\geq 0$. Let $1\leq s<1+\ell/2$ and $u_{0},u_{1}\in H_{(s)}^{-\infty}$ as well as $f\in C([0,T];H_{(s)}^{-\infty})$. Let $u$ be a very weak solution of $H_{(s)}^{-\infty}$-type of \eqref{weak_prob_state}. Then for any regularising families $a_{\varepsilon}$ and $f_{\varepsilon}$ in Definition \ref{veryweak_def}, any representative $(u_{\varepsilon})_{\varepsilon}$ of $u$ converges in $C^{2}([0,T]; H_{(s)}^{-\infty})$ as $\varepsilon\rightarrow0$ to the unique classical
solution in $C^{2}([0,T]; H_{(s)}^{-\infty})$ of the Cauchy problem \eqref{weak_prob_state} given by Part (b) of Theorem \ref{Ruzh_Tar_thm_inhom} (iii).
\end{enumerate}
\end{thm}
Similarly, we can show other consistency \textquote{cases} of Theorem \ref{consist_weak}, corresponding to Part (b) of Theorem \ref{Ruzh_Tar_thm_inhom} (ii) and Part (b) of Theorem \ref{Ruzh_Tar_thm_inhom} (iv):
\begin{thm} \label{consist2_weak} (Consistency-2) Let $\G$ be a graded Lie group and let $\R$ be a positive Rockland operator
of homogeneous degree $\nu$. Let $T>0$.
\begin{enumerate}[label=(\roman*)] \item Let $a(t)\geq a_{0}>0$ and $a\in C^{\alpha}([0,T])$ with $0<\alpha<1$. Let $1\leq s<1+\alpha/(1-\alpha)$, $(u_{0}, u_{1})\in H_{(s)}^{-\infty}$ and $f\in C([0,T];H_{(s)}^{-\infty})$. Let $u$ be a very weak solution of $H_{(s)}^{-\infty}$-type of \eqref{weak_prob_state}. Then for any regularising families $a_{\varepsilon}$ and $f_{\varepsilon}$ in Definition \ref{veryweak_def}, any representative $(u_{\varepsilon})_{\varepsilon}$ of $u$ converges in $C^{2}([0,T]; H_{(s)}^{-\infty})$ as $\varepsilon\rightarrow0$ to the unique classical
solution in $C^{2}([0,T]; H_{(s)}^{-\infty})$ of the Cauchy problem \eqref{weak_prob_state} given by Part (b) of Theorem \ref{Ruzh_Tar_thm_inhom} (ii).

\item Let $a(t)\geq 0$ and $a\in C^{\alpha}([0,T])$ with $0<\alpha<2$. Let $1\leq s<1+\alpha/2$, $u_{0},u_{1}\in H_{(s)}^{-\infty}$ and $f\in C([0,T];H_{(s)}^{-\infty})$. Let $u$ be a very weak solution of $H_{(s)}^{-\infty}$-type of \eqref{weak_prob_state}. Then for any regularising families $a_{\varepsilon}$ and $f_{\varepsilon}$ in Definition \ref{veryweak_def}, any representative $(u_{\varepsilon})_{\varepsilon}$ of $u$ converges in $C^{2}([0,T]; H_{(s)}^{-\infty})$ as $\varepsilon\rightarrow0$ to the unique classical solution in $C^{2}([0,T]; H_{(s)}^{-\infty})$ of the Cauchy problem \eqref{weak_prob_state} given by Part (b) of Theorem \ref{Ruzh_Tar_thm_inhom} (iv).
\end{enumerate}
\end{thm}
The organisation of the paper is as follows. In Section \ref{SEC:prelim} we briefly recall the necessary
concepts of the setting of graded groups. The proof of main results are given in Section \ref{SEC:main_weak} for homogeneous Rockland wave equation. Finally, in Section \ref{SEC:appendix} we briefly discuss the differences in the argument in the case of the inhomogeneous Rockland wave equation.
\section{Preliminaries}
\label{SEC:prelim}

In this section let us very briefly recall the necessary notation concerning the setting of graded groups. We refer to Folland and Stein \cite[Chapter 1]{FS-book}, or to the recent exposition in \cite[Chapter 3]{FR16} for a detailed description of the notions of graded and homogeneous nilpotent Lie groups.

A connected simply connected Lie group $\mathbb{G}$ is called a graded Lie group if its Lie algebra $\mathfrak{g}$ has a vector space decomposition
$$\mathfrak{g}=\bigoplus_{\ell=1}^{\infty}\mathfrak{g}_{\ell},$$
where the $\mathfrak{g}_{\ell}$, $\ell=1,2,...,$ are vector subspaces of $\mathfrak{g}$, all but finitely many equal to $\{0\}$, and satisfying
$$[\mathfrak{g}_{\ell},\mathfrak{g}_{\ell'}]\subset \mathfrak{g}_{\ell+\ell'} \;\;\forall \ell, \ell'\in \mathbb{N}.$$

We fix a basis $\{X_{1},\ldots,X_{n}\}$ of a Lie algebra $\mathfrak{g}$ adapted to the gradation. Then, one can obtain points in $\G$ through the exponential mapping $\exp_{\mathbb{G}}:\mathfrak{g}\rightarrow\mathbb{G}$ as
$$x=\exp_{\mathbb{G}}(x_{1}X_{1}+\ldots+x_{n}X_{n}).$$
Let $A$ be a diagonalisable linear operator on $\mathfrak{g}$ with positive eigenvalues. Then, a family of dilations of a Lie algebra $\mathfrak{g}$ is a family of linear mappings of the form
$$D_{r}={\rm Exp}(A \,{\rm ln}r)=\sum_{k=0}^{\infty}
\frac{1}{k!}({\rm ln}(r) A)^{k}.$$
Here, note that $D_{r}$ is a morphism of $\mathfrak{g}$, i.e.
$$\forall X,Y\in \mathfrak{g},\, r>0,\;
[D_{r}X, D_{r}Y]=D_{r}[X,Y],$$
where $[X,Y]:=XY-YX$ is the Lie bracket. Recall that the dilations can be extended through the exponential mapping to $\G$ by
$$D_{r}(x)=rx:=(r^{\nu_{1}}x_{1},\ldots,r^{\nu_{n}}x_{n}), \;\;x=(x_{1},\ldots,x_{n})\in\mathbb{G},\;\;r>0,$$
where $\nu_{1},\ldots,\nu_{n}$ are weights of the dilations. The homogeneous dimension of $\G$ is denoted by
\begin{equation}\label{dil_weight}
Q:={\rm Tr}\, A=\nu_1+\cdots+\nu_n.
\end{equation}
Let $\widehat{\mathbb{G}}$ be the unitary dual of $\mathbb{G}$.
For a representation $\pi\in\widehat{\mathbb{G}}$, let $\mathcal{H}_{\pi}^{\infty}$ be the space of smooth vectors. We say that a left-invariant differential operator $\mathcal{R}$ on $\mathbb{G}$, which is homogeneous of positive degree, is a Rockland operator, if it satisfies the following Rockland condition:

({\bf R}) for every representation $\pi\in\widehat{\mathbb{G}}$, except for the trivial representation, the operator $\pi(\R)$ is injective on $\mathcal{H}_{\pi}^{\infty}$, i.e.
$$\forall \upsilon \in \mathcal{H}_{\pi}^{\infty}, \;\;\pi(\R)\upsilon=0\Rightarrow \upsilon=0,$$
where $\pi(\R):=d\pi(\R)$ is the infinitesimal representation of $\R$ as of an element of the universal enveloping algebra of $\G$.

For a more detailed discussion of this definition, we refer to \cite[Definition 1.7.4 and Section 4.1.1]{FR16}, that appeared in the work of Rockland \cite{Rockland}. Alternative characterisations of such operators have been obtained by Rockland \cite{Rockland} and Beals \cite{Beals-Rockland}, until the resolution in \cite{HN-79} by Helffer and Nourrigat of the so-called Rockland conjecture, which characterised operators satisfying condition ({\bf R}) as left-invariant homogeneous hypoelliptic differential operators on $\G$.

In this paper we will deal with the Rockland differential operators which are positive in the sense of operators.

We also refer to \cite[Chapter 4]{FR16} for an extensive presentation concerning Rockland operators and their properties, as well as for the consistent development of the corresponding theory of Sobolev spaces. In \cite{CR-CRAS} the corresponding Besov spaces on graded Lie groups and their properties are investigated. Spectral properties of the infinitesimal representations of Rockland operators have been analysed in \cite{tER:97}. For the pseudo-differential calculus on graded Lie groups, we refer to \cite{FR:graded} and \cite{FR16}.

Let $\pi$ be a representation of $\G$ on the separable Hilbert space $\mathcal{H}_{\pi}$. We say that a vector $v\in \mathcal{H}_{\pi}$ is a smooth or of type $C^{\infty}$ if the function
$$\G\ni x \mapsto \pi(x)v\in \mathcal{H}_{\pi}$$
is of class $C^{\infty}$. Let $\mathcal{H}_{\pi}^{\infty}$ be the space of all smooth vectors of a representation $\pi$. Let $\pi$ be a strongly continuous representation of $\G$ on a Hilbert space $\mathcal{H}_{\pi}$. For every $X\in \mathfrak g$ and $v\in \mathcal{H}_{\pi}^{\infty}$ we denote
$$d\pi(X)v:=\lim_{t\rightarrow 0}\frac{1}{t}\left(\pi(\exp_{\G}(tX))v-v\right).$$
Then $d\pi$ is a representation of $\mathfrak{g}$ on $\mathcal{H}_{\pi}^{\infty}$ (see e.g. \cite[Proposition 1.7.3]{FR16}), that is, the
infinitesimal representation associated to $\pi$. By abuse of notation, we will often write $\pi$ instead of $d\pi$, therefore, we write $\pi(X)$ instead of $d\pi(X)$ for any $X\in\mathfrak{g}$.

By the Poincar\'{e}-Birkhoff-Witt theorem, any left-invariant differential operator $T$ on $\G$ can be written in a unique way as a finite sum
\begin{equation}\label{PBW_for}
T=\sum_{|\alpha|\leq M}c_{\alpha}X^{\alpha},
\end{equation}
which allows us to look at $T$ as an element of the universal enveloping algebra $\mathfrak{U}(\mathfrak g)$ of $\mathfrak g$, where all but finitely many of the coefficients $c_{\alpha}\in \mathbb{C}$ are zero and $X^{\alpha}=X_{1}\cdots X_{|\alpha|},$ with $X_{j}\in \mathfrak{g}$. Therefore, the family of infinitesimal representations $\{\pi(T),\pi \in \widehat{\G}\}$ yields a field of operators that turns to be the symbol associated with $T$.

Let $\pi\in \widehat{\G}$ and let $\R$ be a positive Rockland operator of homogeneous degree $\nu>0$. Then, from \eqref{PBW_for} we obtain the following infinitesimal representation of $\R$ associated to $\pi$,
$$\pi(\R)=\sum_{[\alpha]=\nu}c_{\alpha}\pi(X)^{\alpha},$$
where $\pi(X)^{\alpha}=\pi(X^{\alpha})=\pi(X_{1}^{\alpha_{1}}\cdots X_{n}^{\alpha_{n}})$ and $[\alpha]=\nu_{1}\alpha_{1}+\cdots+\nu_{n}\alpha_{n}$ is the homogeneous degree of the multiindex $\alpha$, with $X_{j}$ being homogeneous of degree $\nu_{j}$.

Recall that $\R$ and $\pi(\R)$ are densely defined on $\mathcal{D}(\G)\subset L^{2}(\G)$ and $\mathcal{H}_{\pi}^{\infty}\subset \mathcal{H}_{\pi}$, respectively (see e.g. \cite[Proposition 4.1.15]{FR16}). Let us denote the self-adjoint extension of $\R$ on $L^{2}(\G)$ by $\R_{2}$ and keep the same notation $\pi(\R)$ for the self-adjoint extensions on $\mathcal{H}_{\pi}$ of the infinitesimal representations.

By the spectral theorem for unbounded operators \cite[Theorem VIII.6]{RS80}, we write
$$\R_{2}=\int_{\mathbb{R}}\lambda dE(\lambda)\;\;\text{and}\;\;\pi(\R)=\int_{\mathbb{R}}\lambda dE_{\pi}(\lambda),$$
where $E$ and $E_{\pi}$ are the spectral measures corresponding to $\R_{2}$ and $\pi(\R)$.

Moreover, for any $f\in L^{2}(\G)$ one has
\begin{equation}\label{prel_for_1}
\mathcal{F} (\phi(\R)f)(\pi)=\phi(\pi(\R))\widehat{f}(\pi),
\end{equation}
for any measurable bounded function $\phi$ on the real line $\mathbb{R}$ (see e.g. \cite[Corollary 4.1.16]{FR16}). Note that the
infinitesimal representations $\pi(\R)$ of a positive Rockland operator $\R$ are also positive, because of the relations between their spectral measures. In \cite{HJL85} Hulanicki, Jenkins and Ludwig showed that the spectrum of $\pi(\R)$, with $\pi \in \widehat{\G}\backslash \{1\}$, is discrete and lies in $(0,\infty)$, which allows us to choose an orthonormal basis for $\mathcal{H}_{\pi}$ such that the infinite matrix associated to the self-adjoint operator $\pi(\R)$ has the following form
\begin{equation}\label{pi_R_matrix}
\pi(\R)=\begin{pmatrix} \pi_{1}^{2} & 0 & \ldots &\ldots \\
0 & \pi_{2}^{2} & 0 & \ldots\\
\vdots & 0 & \ddots &\\
\vdots & \vdots &  & \ddots \end{pmatrix},
\end{equation}
where $\pi\in \widehat{\G}\backslash \{1\}$ and $\pi_{j}\in \mathbb{R}_{>0}$.

Now, since we will also deal with the Fourier transform on $\G$, let us briefly recall it.

As usual we identify irreducible unitary representations with their equivalence classes. For $f\in L^{1}(\G)$ and $\pi \in \widehat{G}$, the group Fourier transform of $f$ at $\pi$ is defined by
$$\mathcal{F}_{\G}f(\pi)\equiv \widehat{f}(\pi)\equiv \pi(f):=\int_{\G}f(x)\pi(x)^{*}dx,$$
with integration against the biinvariant Haar measure on $\G$. It implies a linear mapping $\widehat{f}(\pi)$ from the Hilbert space $\mathcal{H}_{\pi}$ to itself that can be represented by an infinite matrix once we choose a basis for $\mathcal{H}_{\pi}$. Consequently, we have
$$\mathcal{F}_{\G}(\R f)(\pi)=\pi(\R)\widehat{f}(\pi).$$
In the sequel, when we write $\widehat{f}(\pi)_{m,k}$, we will be using the same basis in the representation space $\mathcal{H}_{\pi}$ as the one giving \eqref{pi_R_matrix}.

By Kirillov's orbit method (see e.g. \cite{CG90}), we know that the Plancherel measure $\mu$ on the dual $\widehat{\G}$ can be constructed explicitly. In particular, this means that we can have the Fourier inversion formula. Furthermore, the operator $\pi(f)=\widehat{f}(\pi)$ is Hilbert-Schmidt, i.e.
$$\|\pi(f)\|^{2}_{{\rm HS}}={\rm Tr}(\pi(f)\pi(f)^{*})<\infty,$$
and the function $\widehat{\G}\ni \pi \mapsto \|\pi(f)\|^{2}_{{\rm HS}}$ is integrable with respect to $\mu$. Moreover, the Plancherel formula holds (see e.g. \cite{CG90} or \cite{FR16}):
\begin{equation}\label{planch_for}
\int_{\G}|f(x)|^{2}dx=\int_{\widehat{\G}}\|\pi(f)\|^{2}_{{\rm HS}}d\mu(\pi).
\end{equation}
\section{Proofs of main results}
\label{SEC:main_weak}
In this section we prove main results of the paper when $f\equiv0$, and in the case when $f\not\equiv 0$ we refer to Section \ref{SEC:appendix} for the differences in the argument in this case.

First, we need to prove the following result:
\begin{lem}\label{lem_useful} A functional $u$ belongs to $H_{s}^{-\infty}$ if and only if for any positive $\delta>0$ there exists a positive constant $C_{\delta}>0$ such that
$$|\widehat{u}(\pi)_{m,k}|\leq C_{\delta}e^{\delta \pi_{m}^{\frac{1}{s}}}$$
holds for all $\pi \in \widehat {\G}$ and any $m,k\in \mathbb{N}$, where $\pi_{m}$ are strictly positive real numbers from \eqref{pi_R_matrix}. Similarly, $u\in H_{(s)}^{-\infty}$ holds if and only if there exist positive constants $\eta>0$ and $C>0$ such that
$$|\widehat{u}(\pi)_{m,k}|\leq Ce^{\eta \pi_{m}^{\frac{1}{s}}}$$
holds for all $\pi \in \widehat {\G}$, and any $m,k\in \mathbb{N}$.
\end{lem}
\begin{proof}[Proof of Lemma \ref{lem_useful}] Using Plancherel's identity \eqref{planch_for}, we can characterise the Gevrey Roumieu ultradistributions $H_{s}^{-\infty}$ and the Gevrey Beurling ultradistributions $H_{(s)}^{-\infty}$ by
\begin{equation*}
\begin{split}
u\in H_{s}^{-\infty}\Leftrightarrow \forall \delta>0:
\|e^{-\delta \R^{\frac{1}{2s}}}u\|^{2}_{L^{2}(\G)}&=\int_{\widehat{\G}}\|e^{-\delta\pi(\R)^{\frac{1}{2s}}}\widehat{u}(\pi)\|^{2}_{{\rm HS}}d\mu(\pi)
\\&=
\int_{\widehat{\G}}\sum_{m,k}|e^{-\delta \pi_{m}^{\frac{1}{s}}}\widehat{u}(\pi)_{m,k}|^{2}d\mu(\pi)
<\infty,
\end{split}
\end{equation*}
and
\begin{equation*}
\begin{split}
u\in H_{(s)}^{-\infty}\Leftrightarrow \exists \eta>0:
\|e^{-\eta \R^{\frac{1}{2s}}}u\|^{2}_{L^{2}(\G)}&=\int_{\widehat{\G}}\|e^{-\eta\pi(\R)^{\frac{1}{2s}}}\widehat{u}(\pi)\|^{2}_{{\rm HS}}d\mu(\pi)
\\&=
\int_{\widehat{\G}}\sum_{m,k}|e^{-\eta \pi_{m}^{\frac{1}{s}}}\widehat{u}(\pi)_{m,k}|^{2}d\mu(\pi)
<\infty,
\end{split}
\end{equation*}
respectively.
\end{proof}
We prove Theorem \ref{Ruzh_Tar_thm_inhom} (i) in Section \ref{SEC:appendix}. Now let us prove Part (b) of Theorem \ref{Ruzh_Tar_thm_inhom} (ii), (iii) and (iv).
\begin{proof}[Proof of Part (b) of Theorem \ref{Ruzh_Tar_thm_inhom} (ii), (iii) and (iv)] Since the way of deriving Parts (b) of Theorem \ref{Ruzh_Tar_thm_inhom} (ii), (iii), (iv) from Parts (ii), (iii), (iv) of Theorem \ref{Ruzh_Tar_thm}, respectively, is similar, let us show it only for Part (b) of Theorem \ref{Ruzh_Tar_thm_inhom} (iii), which will be useful in investigating the weak solution of \eqref{intro_eq1}. Recall the characterisation of $H_{(s)}^{-\infty}$. Since $u_{0},u_{1}\in H_{(s)}^{-\infty}$ and by Lemma \ref{lem_useful} we see that there exist positive constants $A_{1}$ and $C_{1}$ such that
\begin{equation}\label{eq_main1}
\begin{split}
&\left|e^{-A_{1}\pi_{m}^{\frac{1}{s}}}\widehat{u_{0}}(\pi)_{m,k}\right|\leq C_{1},\\
&\left|e^{-A_{1}\pi_{m}^{\frac{1}{s}}}\widehat{u_{1}}(\pi)_{m,k}\right|\leq C_{1}
\end{split}
\end{equation}
for all $m,k\in \mathbb{N}$. By the proof of \cite[Case 3 of Theorem 1.1, Page 20]{RT17a}, we know that there exist positive constants $C$ and $K$ such that
\begin{equation}\label{Tar_for1}
|\pi_{m}\widehat{u}(t,\pi)_{m,k}|^{2}\leq Ce^{K^{\prime}\pi_{m}^{\frac{1}{s}}}(|\widehat{u}_{0}(\pi)_{m,k}|^{2}+|\widehat{u}_{1}(\pi)_{m,k}|^{2})
\end{equation}
for $1\leq s<\sigma=1+\ell/2$ and some $K^{\prime}>0$ small enough, and all $m,k\in \mathbb{N}$, where $\pi_{m}$ are strictly positive real numbers from \eqref{pi_R_matrix}.

Putting \eqref{eq_main1} into \eqref{Tar_for1} we obtain that for all $t\in [0,T]$ there exist positive constants $C_{2}$ and $A_{2}$ such that
$$\left|e^{-A_{2}\pi_{m}^{\frac{1}{s}}}\widehat{u}(t,\pi)_{m,k}\right|^{2}\leq C_{2},$$
which implies that there exist positive constants $A$ and $C>0$ such that
$$|\widehat{u}(t,\pi)_{m,k}|\leq Ce^{A\pi_{m}^{\frac{1}{s}}},$$
that is, $u(t,\cdot)\in H^{-\infty}_{(s)}$ provided that
$$1\leq s<\sigma=1+\frac{\ell}{2}.$$
This completes the proof of Part (b) of Theorem \ref{Ruzh_Tar_thm_inhom} (iii).

Similarly, Parts (ii) and (iv) of Theorem \ref{Ruzh_Tar_thm} imply Part (b) of Theorem \ref{Ruzh_Tar_thm_inhom} (ii) and Part (b) of Theorem \ref{Ruzh_Tar_thm_inhom} (iv), respectively.
\end{proof}

\begin{proof}[Proof of Theorem \ref{exist_weak}] (i) Assume that the coefficient $a=a(t)$ is a distribution with compact support contained in $[0,T]$. Then, we note that the formulation of \eqref{weak_prob_state} might be impossible in the distributional sense due to issues related to the product of distributions. Therefore, we replace \eqref{weak_prob_state} with a regularised equation. Namely, if we regularise the coefficient $a$ by a convolution with a mollifier in $C^{\infty}_{0}(\mathbb{R})$, then we get nets of smooth functions as coefficients. For this, we take $\psi \in C^{\infty}_{0}(\mathbb{R})$, $\psi\geq 0$ with $\int\psi=1$, and $\omega(\varepsilon)>0$ such that $\omega(\varepsilon) \rightarrow0$ as $\varepsilon\rightarrow 0$ to be chosen later. Then, we define $\psi_{\omega_{\varepsilon}}$ and $a_{\varepsilon}$ by
$$\psi_{\omega_{\varepsilon}}(t):=\frac{1}{\omega(\varepsilon)}\psi\left(\frac{t}{\omega(\varepsilon)}\right)$$
and
$$a_{\varepsilon}(t):=(a\ast\psi_{\omega(\varepsilon)})(t)$$
for all $t\in [0,T]$, respectively. Using these representations of $\psi_{\omega_{\varepsilon}}$ and $a_{\varepsilon}$ and identifying the measure $a(t)$ with its density, we get
\begin{equation*}
\begin{split}
a_{\varepsilon}(t)&=(a\ast \psi_{\omega(\varepsilon)})(t)=\int_{\mathbb{R}}a(t-\tau)\psi_{\omega(\varepsilon)}(\tau)d\tau=
\int_{\mathbb{R}}a(t-\omega(\varepsilon)\tau)\psi(\tau)d\tau\\&
=\int_{K}a(t-\omega(\varepsilon)\tau)\psi(\tau)d\tau\geq a_{0}\int_{K}\psi(\tau)d\tau:=\widetilde{a_{0}}>0,
\end{split}
\end{equation*}
where we have used that $a(t)$ is a positive distribution with compact support (hence a Radon measure) and $\psi \in C^{\infty}_{0}(\mathbb{R})$, ${\rm supp}\;\psi\subset K$, $\psi\geq0$ in above. Here, note that $\widetilde{a_{0}}$ does not depend on $\varepsilon$.

We also note that by virtue of the structure theorem for compactly supported distributions, there exist a natural number $L$ and positive constant $c$ such that for all $k\in \mathbb{N}_{0}$ and $t\in[0,T]$ we have
\begin{equation}\label{for_weak_eq1}
|\partial_{t}^{k}a_{\varepsilon}(t)|\leq c(\omega(\varepsilon))^{-L-k}.
\end{equation}
Thus, $a_{\varepsilon}$ is $C^{\infty}$-moderate regularisation of the coefficient $a(t)$ under appropriate conditions on $\omega(\varepsilon)$, then fixing $\varepsilon \in (0,1]$ we consider the following regularised problem
\begin{equation}\label{for_weak_eq2}
\left\{
                \begin{array}{ll}
                  \partial_{t}^{2}u_{\varepsilon}(t)+a_{\varepsilon}(t)\R u_{\varepsilon}(t)=0, \;t\in [0,T],\\
                  u_{\varepsilon}(0)=u_{0}\in L^{2}(\G),\\
                  \partial_{t}u_{\varepsilon}(0)=u_{1}\in L^{2}(\G),
                \end{array}
              \right.
\end{equation}
where $(u_{0},u_{1})\in H_{\R}^{s+\frac{\nu}{2}}\times H_{\R}^{s}$, $a_{\varepsilon}\in C^{\infty}[0,T]$. Then, Theorem \ref{Ruzh_Tar_thm} (i) implies that the regularised problem \eqref{for_weak_eq2} has a unique solution in the space $C([0,T];H_{\R}^{s+\frac{\nu}{2}})\cap C^{1}([0,T];H_{\R}^{s})$. Actually, noting $a_{\varepsilon}\in C^{\infty}([0,T])$ and differentiating both sides of the equation \eqref{for_weak_eq2} in $t$ inductively, one can see that this unique solution is from $C^{\infty}([0,T];H_{\R}^{s})$.

Since we will use some arguments of the proof of Theorem \ref{Ruzh_Tar_thm} (i), we refer to {\cite[the proof of Case 1 of Theorem 1.1]{RT17a}} or the proof of Theorem \ref{Ruzh_Tar_thm_inhom} (i) in Section \ref{SEC:appendix}, which is an inhomogeneous version of Theorem \ref{Ruzh_Tar_thm} (i) (i.e. an inhomogeneous version of {\cite[Theorem 1.1]{RT17a}}).

Thus, recalling
\begin{equation*}S(t):=\left(
                   \begin{array}{cc}
                     2a(t) & 0 \\
                     0 & 2 \\
                   \end{array}
                 \right),
\end{equation*}
by the proof of {\cite[Case 1 of Theorem 1.1]{RT17a}} (or Theorem \ref{Ruzh_Tar_thm_inhom} (i) with $f\equiv0$), and noting \eqref{for_weak_eq1}, we get
$$\|\partial_{t}S(t)\|\leq C|\partial_{t}a_{\varepsilon}(t)|\leq C\omega(\varepsilon)^{-L-1}.$$
Then, {\cite[Formula (3.8)]{RT17a}} (or \eqref{estimating_E(t)} with $f\equiv0$), Gronwall's lemma and {\cite[the first formula formula after Formula (3.9)]{RT17a}}, {\cite[Formulae (4.4)-(4.6)]{RT17a}} (or \eqref{V_estim}-\eqref{app_eq4} with $f\equiv0$) imply that
\begin{multline}\label{for_weak_eq3}
    \|u_{\varepsilon}(t,\cdot)\|^{2}_{ H^{s+\frac{\nu}{2}}_{\R}(\G)}+\|\partial_{t} u_{\varepsilon}(t,\cdot)\|^{2}_{ H^{s}_{\R}(\G)}\\
\leq C\exp(c\omega(\varepsilon)^{-L-1}T)(\|u_{0}\|^{2}_{H^{s+\frac{\nu}{2}}_{\R}(\G)}+\|u_{1}\|^{2}_{H^{s}_{\R}(\G)}).
    \end{multline}
If we take $(\omega(\varepsilon))^{-L-1}\approx \log \varepsilon$, then \eqref{for_weak_eq3} becomes
$$\|u_{\varepsilon}(t,\cdot)\|^{2}_{ H^{s+\frac{\nu}{2}}_{\R}(\G)}+\|\partial_{t} u_{\varepsilon}(t,\cdot)\|^{2}_{ H^{s}_{\R}(\G)}
\leq C\varepsilon^{-L-1}(\|u_{0}\|^{2}_{H^{s+\frac{\nu}{2}}_{\R}(\G)}+\|u_{1}\|^{2}_{H^{s}_{\R}(\G)}),$$
with possibly new constant $L$.

Now, to obtain that $u_{\varepsilon}$ is $C^{\infty}([0,T];H_{\R}^{s})$-moderate, we need to show that for all $t\in[0,T]$ and $\varepsilon\in (0,1]$,
\begin{equation}\label{need_to_prov}
\|\partial_{t}u_{\varepsilon}(t,\cdot)\|_{H_{\R}^{s}}\leq C\varepsilon^{-L-1},
\|u_{\varepsilon}(t,\cdot)\|_{H_{\R}^{s+\frac{\nu}{2}}}\leq C\varepsilon^{-L}
\end{equation}
hold for some $L>0$. Indeed, once we prove this, then acting by the iterations of $\partial_{t}$ on the equality
$$\partial_{t}^{2}u_{\varepsilon}(t)=-a_{\varepsilon}(t)\R u_{\varepsilon}(t),$$
and taking it in $L^{2}(\G)$-norms, we conclude that $u_{\varepsilon}$ is $C^{\infty}([0,T];H_{\R}^{s})$-moderate. In order to show \eqref{need_to_prov}, we apply {\cite[Formula (3.6)]{RT17a}} and {\cite[the first formula formula afer Formula (3.9)]{RT17a}} (or \eqref{two_side_est_E_2} and \eqref{V_estim} with $f\equiv0$) to $u_{\varepsilon}$, and then by the properties of $a_{\varepsilon}$, we arrive at
$$\pi_{m}^{2}|\widehat{u}_{\varepsilon}(t,\pi)_{m,k}|^{2}+|\partial_{t}\widehat{u}_{\varepsilon}(t,\pi)_{m,k}|^{2}\\
\leq C \varepsilon^{-L-1}(\pi_{m}^{2}|\widehat{u}_{0}(\pi)_{m,k}|^{2}+|\widehat{u}_{1}(\pi)_{m,k}|^{2})
$$
for all $t\in[0,T]$, $\pi\in \widehat{G}$, $m,k\in \mathbb{N}$ and for some $L>0$, where the constant $C$ does not depend on $\pi$. Thus, multiplying this by appropriate powers of $\pi_{m}$ we obtain \eqref{need_to_prov}.

Since $u_{\varepsilon}$ is $C^{\infty}([0,T];H_{\R}^{s})$-moderate, by the Definition \ref{veryweak_def} we conclude that the Cauchy problem \eqref{weak_prob_state} has a very weak solution.

Now we prove Part (ii). Similarly as in Part (i), in this case one get that for $a_{\varepsilon}(t)\geq0$ there are $L\in \mathbb{N}$ and $c_{1}>0$ such that
\begin{equation}\label{for_weak_eq6} |\partial_{t}^{k}a_{\varepsilon}(t)|\leq c_{1}(\omega(\varepsilon))^{-L-k},
\end{equation}
for all $k\in \mathbb{N}_{0}$ and $t\in [0,T]$, which means that $a_{\varepsilon}(t)$ is a $C^{\infty}$-moderate regularisation of $a(t)$. Then, fixing $\varepsilon \in (0,1]$, we consider the following regularised problem
\begin{equation}\label{for_weak_eq7}
\left\{
                \begin{array}{ll}
                  \partial_{t}^{2}u_{\varepsilon}(t)+a_{\varepsilon}(t)\R u_{\varepsilon}(t)=0, \;t\in [0,T],\\
                  u_{\varepsilon}(0)=u_{0}\in L^{2}(\G),\\
                  \partial_{t}u_{\varepsilon}(0)=u_{1}\in L^{2}(\G),
                \end{array}
              \right.
\end{equation}
where $u_{0},u_{1}\in H_{(s)}^{-\infty}$ and $a_{\varepsilon}\in C^{\infty}[0,T]$. Then, we can use Part (b) of Theorem \ref{Ruzh_Tar_thm_inhom} (iii), which implies that the equation \eqref{for_weak_eq7} has a unique solution in the space $u\in C^{2}([0,T];H_{(s)}^{-\infty})$ for any $s$. Actually, this unique solution is from $C^{\infty}([0,T];H_{(s)}^{-\infty})$, which can be checked by differentiating both sides of the equation \eqref{for_weak_eq7} in $t$ inductively noting that $a_{\varepsilon}\in C^{\infty}([0,T])$. Applying Part (b) of Theorem \ref{Ruzh_Tar_thm_inhom} (iii) to the equation \eqref{for_weak_eq7}, using the inequality
$$|\partial_{t}a_{\varepsilon}(t)|\leq C(\omega(\varepsilon))^{-L-1},$$
we have
\begin{equation}\label{for_weak_eq8}
|\pi_{m}\widehat{u}_{\varepsilon}(t,\pi)_{m,k}|^{2}+|\partial_{t}\widehat{u}_{\varepsilon}(t,\pi)_{m,k}|^{2}\leq Ce^{K^{\prime}(\omega(\varepsilon))^{-L-1}\pi_{m}^{\frac{1}{s}}}(|\widehat{u}_{0}(\pi)_{m,k}|^{2}+|\widehat{u}_{1}(\pi)_{m,k}|^{2})
\end{equation}
for all $m,k\in \mathbb{N}$. Taking $\omega^{-1}(\varepsilon)\approx(\log \varepsilon)^{r}$ for an appropriate $r$, and repeating as in the proof of Part (b) of Theorem \ref{Ruzh_Tar_thm_inhom} (iii), from \eqref{for_weak_eq8} we obtain that there exists a positive $\eta$ and, for $p=0,1$ there are $c_{p},N_{p}>0$ such that
\begin{equation}\label{for_weak_eq9}
\|e^{-\eta \R^{\frac{1}{2s}}}\partial_{t}^{p}u_{\varepsilon}(t,\cdot)\|_{L^{2}(\G)}\leq c_{p}\varepsilon^{-N_{p}-p},
\end{equation}
for all $t\in [0,T]$ and $\varepsilon\in (0,1]$. Now, in order to prove that \eqref{for_weak_eq9} holds for all $p\in \mathbb{N}$, we use the following equality
$$\partial_{t}^{2}u_{\varepsilon}(t)=-a_{\varepsilon}(t)\R u_{\varepsilon}(t).$$
Namely, acting by the iterations of $\partial_{t}$ and taking into account the properties of $a_{\varepsilon}(t)$, from \eqref{for_weak_eq9}, we conclude that $u_{\varepsilon}$ is $C^{\infty}([0,T];H_{(s)}^{-\infty})$-moderate.

This completes the proof of Theorem \ref{exist_weak}.
\end{proof}
Now we prove Theorem \ref{uniq_exist}.
\begin{proof}[Proof of Theorem \ref{uniq_exist}] (i) We assume that the Cauchy problem has another solution $v\in\mathcal{G}([0,T];H_{\R}^{s})$. At the level of representatives this means
\begin{equation}\label{for_weak_eq_uniq1}
\left\{
                \begin{array}{ll}
                  \partial_{t}^{2}(u_{\varepsilon}-v_{\varepsilon})(t)+a_{\varepsilon}(t)\R (u_{\varepsilon}-v_{\varepsilon})(t)=\rho_{\varepsilon}(t), \;t\in [0,T],\\
                  (u_{\varepsilon}-v_{\varepsilon})(0)=0,\\
                  (\partial_{t}u_{\varepsilon}-\partial_{t}v_{\varepsilon})(0)=0,
                \end{array}
              \right.
\end{equation}
with $\rho_{\varepsilon}=(\widetilde{a_{\varepsilon}}(t)-a_{\varepsilon}(t))\R v_{\varepsilon}(t)$, where $(\widetilde{a_{\varepsilon}})_{\varepsilon}$ is an approximation corresponding to $v_{\varepsilon}$. Since $(a_{\varepsilon})_{\varepsilon}\sim(\widetilde{a_{\varepsilon}})_{\varepsilon}$, we have that $\rho_{\varepsilon}$ is $C^{\infty}([0,T];H_{\R}^{s})$-negligible. Let us write this as in the following first order system
$$\partial_{t}\begin{pmatrix} W_{1,\varepsilon} \\ W_{2,\varepsilon} \end{pmatrix}
=\begin{pmatrix} 0 & i\R^{1/2}\\ ia_{\varepsilon}(t)\R^{1/2} & 0 \end{pmatrix}
\begin{pmatrix} W_{1,\varepsilon} \\ W_{2,\varepsilon} \end{pmatrix}+\begin{pmatrix} 0 \\ \rho_{\varepsilon} \end{pmatrix},
$$
where
$$W_{1,\varepsilon}:=i\R^{1/2}(u_{\varepsilon}-v_{\varepsilon})\;\;\text{and}\;\;W_{2,\varepsilon}:=\partial_{t}(u_{\varepsilon}-v_{\varepsilon}).$$
This system will be studied after the group Fourier transform, as a system of the type
$$\partial_{t}V_{\varepsilon}(t,\pi)_{m,k}=i\pi_{m}A_{\varepsilon}(t,\pi)_{m,k}V_{\varepsilon}(t,\pi)_{m,k}+P_{\varepsilon}(t,\pi)_{m,k},$$
for all $\pi \in \widehat{\G}$, and any $m,k\in \mathbb{N}$, where $P_{\varepsilon}(t,\pi)_{m,k}=\begin{pmatrix} 0 \\ \mathcal (F_{\G}\rho_{\varepsilon})_{m,k} \end{pmatrix}$, $A_{\varepsilon}(t,\pi)_{m,k}=\begin{pmatrix} 0 & 1\\ a_{\varepsilon}(t) & 0 \end{pmatrix},$ and $V_{\varepsilon}(0,\pi)_{m,k}=\begin{pmatrix} 0 \\ 0 \end{pmatrix}$. We define the energy
$$E_{\varepsilon}(t,\pi)_{m,k}:=(S_{\varepsilon}(t,\pi)_{m,k}V_{\varepsilon}(t,\pi)_{m,k},V_{\varepsilon}(t,\pi)_{m,k})$$
for the symmetriser $S_{\varepsilon}(t,\pi)_{m,k}=\begin{pmatrix} a_{\varepsilon}(t) & 0\\ 0 & 1 \end{pmatrix}$. Since $a_{\varepsilon}(t)$ is continuous, then from the definition of the energy we get
\begin{equation}\label{est_ener_uniq}
c_{0}|V_{\varepsilon}(t,\pi)_{m,k}|^{2}\leq E_{\varepsilon}(t,\pi)_{m,k} \leq c_{1}|V_{\varepsilon}(t,\pi)_{m,k}|^{2}
\end{equation}
for some positive constants $c_{0}$ and $c_{1}$. Then, a direct calculation gives that
\begin{equation}\label{formual_1_uniq_11}
\begin{split}
\partial_{t}&E_{\varepsilon}(t,\pi)_{m,k}\\&
=(\partial_{t}S_{\varepsilon}(t,\pi)_{m,k}V_{\varepsilon}(t,\pi)_{m,k},V_{\varepsilon}(t,\pi)_{m,k})
+(S_{\varepsilon}(t,\pi)_{m,k}\partial_{t}V_{\varepsilon}(t,\pi)_{m,k},V_{\varepsilon}(t,\pi)_{m,k})\\&
+(S_{\varepsilon}(t,\pi)_{m,k}V_{\varepsilon}(t,\pi)_{m,k},\partial_{t}V_{\varepsilon}(t,\pi)_{m,k})\\&
=(\partial_{t}S_{\varepsilon}(t,\pi)_{m,k}V_{\varepsilon}(t,\pi)_{m,k},V_{\varepsilon}(t,\pi)_{m,k})\\&
+i\pi_{m}(S_{\varepsilon}(t,\pi)_{m,k}A_{\varepsilon}(t,\pi)_{m,k}V_{\varepsilon}(t,\pi)_{m,k},V_{\varepsilon}(t,\pi)_{m,k})\\&
-i\pi_{m}(S_{\varepsilon}(t,\pi)_{m,k}V_{\varepsilon}(t,\pi)_{m,k},A_{\varepsilon}(t,\pi)_{m,k}V_{\varepsilon}(t,\pi)_{m,k})\\&
+(S_{\varepsilon}(t,\pi)_{m,k}P_{\varepsilon}(t,\pi)_{m,k},V_{\varepsilon}(t,\pi)_{m,k})
+(S_{\varepsilon}(t,\pi)_{m,k}V_{\varepsilon}(t,\pi)_{m,k},P_{\varepsilon}(t,\pi)_{m,k})\\&
=(\partial_{t}S_{\varepsilon}(t,\pi)_{m,k}V_{\varepsilon}(t,\pi)_{m,k},V_{\varepsilon}(t,\pi)_{m,k})
\\&+i\pi_{m}((S_{\varepsilon}A_{\varepsilon}-A^{*}_{\varepsilon}S_{\varepsilon})(t,\pi)_{m,k}V_{\varepsilon}(t,\pi)_{m,k},
V_{\varepsilon}(t,\pi)_{m,k})\\&
+(S_{\varepsilon}(t,\pi)_{m,k}P_{\varepsilon}(t,\pi)_{m,k},V_{\varepsilon}(t,\pi)_{m,k})
+(V_{\varepsilon}(t,\pi)_{m,k},S_{\varepsilon}(t,\pi)_{m,k}P_{\varepsilon}(t,\pi)_{m,k})\\&
=(\partial_{t}S_{\varepsilon}(t,\pi)_{m,k}V_{\varepsilon}(t,\pi)_{m,k},V_{\varepsilon}(t,\pi)_{m,k})+
2{\rm Re}(S_{\varepsilon}(t,\pi)_{m,k}P_{\varepsilon}(t,\pi)_{m,k},V_{\varepsilon}(t,\pi)_{m,k})\\&
\leq \|\partial_{t}S_{\varepsilon}\||V_{\varepsilon}(t,\pi)_{m,k}|^{2}+2\|S_{\varepsilon}\||P_{\varepsilon}(t,\pi)_{m,k}|
|V_{\varepsilon}(t,\pi)_{m,k}|,
\end{split}
\end{equation}
where we have used $(S_{\varepsilon}A_{\varepsilon}-A^{*}_{\varepsilon}S_{\varepsilon})(t,\pi)_{m,k}=0$.
In the case when $|V_{\varepsilon}(t,\pi)_{m,k}|\geq 1$, taking into account \eqref{est_ener_uniq} we obtain from above that
\begin{equation}\label{for_weak_eq_uniq2}
\begin{split}
\partial_{t}E_{\varepsilon}(t,\pi)_{m,k}&\leq \|\partial_{t}S_{\varepsilon}\||V_{\varepsilon}(t,\pi)_{m,k}|^{2}
+2\|S_{\varepsilon}\||P_{\varepsilon}(t,\pi)_{m,k}|
|V_{\varepsilon}(t,\pi)_{m,k}|\\&
\leq (\|\partial_{t}S_{\varepsilon}\|+2\|S_{\varepsilon}\||P_{\varepsilon}(t,\pi)_{m,k}|)|V_{\varepsilon}(t,\pi)_{m,k}|^{2}\\&
\leq (|\partial_{t}a_{\varepsilon}(t)|+2|a_{\varepsilon}(t)||P_{\varepsilon}(t,\pi)_{m,k}|)|V_{\varepsilon}(t,\pi)_{m,k}|^{2}\\&
\leq c(\omega(\varepsilon))^{-L-1}E_{\varepsilon}(t,\pi)_{m,k}
\end{split}
\end{equation}
for some constant $c>0$. Then, the Gronwall lemma implies that
$$E_{\varepsilon}(t,\pi)_{m,k}\leq\exp(c(\omega(\varepsilon))^{-L-1}T)E_{\varepsilon}(0,\pi)_{m,k}$$
for all $T>0$.
Hence, by \eqref{est_ener_uniq} we obtain for the constant $c_{1}$ independent of $t\in [0,T]$ and $\pi$ that
\begin{equation*}
\begin{split}
c_{0}|V_{\varepsilon}(t,\pi)_{m,k}|^{2}\leq E_{\varepsilon}(t,\pi)_{m,k}&\leq \exp(c(\omega(\varepsilon))^{-L-1}T)E_{\varepsilon}(0,\pi)_{m,k}\\&
\leq \exp(c_{1}(\omega(\varepsilon))^{-L-1}T)|V_{\varepsilon}(0,\pi)_{m,k}|^{2}.
\end{split}
\end{equation*}
Choosing $(\omega(\varepsilon))^{-L-1}\approx\log\varepsilon$, we get
$$|V_{\varepsilon}(t,\pi)_{m,k}|^{2}\leq c\varepsilon^{-L-1}|V_{\varepsilon}(0,\pi)_{m,k}|^{2}$$
for some positive constant $c$ and some (new) $L$. It implies for all $\pi$, and any $m,k\in \mathbb{N}$ and $t\in [0,T]$ that
$$|V_{\varepsilon}(t,\pi)_{m,k}|\equiv0,$$
since $|V_{\varepsilon}(0,\pi)_{m,k}|=0$.

Now let us consider the case $|V_{\varepsilon}(t,\pi)_{m,k}|<1$. Assume that $$|V_{\varepsilon}(t,\pi)_{m,k}|\geq c(\omega(\varepsilon))^{\alpha}$$ for some constant $c$ and $\alpha>0$, i.e.
\begin{equation}\label{Vless1_trick}
\frac{1}{|V_{\varepsilon}(t,\pi)_{m,k}|}\leq C(\omega(\varepsilon))^{-\alpha}.
\end{equation}
In this case, from \eqref{Vless1_trick} noting $$|V_{\varepsilon}(t,\pi)_{m,k}|=\frac{|V_{\varepsilon}(t,\pi)_{m,k}|^{2}}{|V_{\varepsilon}(t,\pi)_{m,k}|}\leq C(\omega(\varepsilon))^{-\alpha}|V_{\varepsilon}(t,\pi)_{m,k}|^{2}$$ and \eqref{est_ener_uniq}, we get from \eqref{formual_1_uniq_11} the following energy estimate
$$\partial_{t}E_{\varepsilon}(t,\pi)_{m,k}\leq C(\omega(\varepsilon))^{-L_{1}}E_{\varepsilon}(t,\pi)_{m,k},$$
where $L_{1}=L+\max\{1,\alpha\}$. Again applying the Gronwall lemma, we arrive at
$$|V_{\varepsilon}(t,\pi)_{m,k}|^{2}\leq \exp(C^{\prime}(\omega(\varepsilon))^{-L_{1}}T)|V_{\varepsilon}(0,\pi)_{m,k}|^{2}.$$
Then, taking $(\omega(\varepsilon))^{-L_{1}}\approx\log\varepsilon$, it follows that
$$|V_{\varepsilon}(t,\pi)_{m,k}|^{2}\leq c^{\prime}\varepsilon^{-L_{1}}|V_{\varepsilon}(0,\pi)_{m,k}|^{2}$$
for some $c^{\prime}$ and some (new) $L_{1}$, which implies
$$|V_{\varepsilon}(t,\pi)_{m,k}|\equiv0$$
for all $\pi$ and $t\in [0,T]$, since we have $|V_{\varepsilon}(0,\pi)_{m,k}|=0$.

The case $|V_{\varepsilon}(t,\pi)_{m,k}|\leq c(\omega(\varepsilon))^{\alpha}$ for some constant $c$ and $\alpha>0$ is trivial. Thus, the first part is proved.

(ii) We prove this part in the similar way as Part (i) but using the quasi-symmetrisers. We assume that the Cauchy problem has another solution \\ $v\in \mathcal{G}([0,T];H_{(s)}^{-\infty})$. At the level of representatives this means that
\begin{equation}\label{for_weak_eq_uniq3}
\left\{
                \begin{array}{ll}
                  \partial_{t}^{2}(u_{\varepsilon}-v_{\varepsilon})(t)+a_{\varepsilon}(t)\R (u_{\varepsilon}-v_{\varepsilon})(t)=\rho_{\varepsilon}(t), \;t\in [0,T],\\
                  (u_{\varepsilon}-v_{\varepsilon})(0)=0,\\
                  (\partial_{t}u_{\varepsilon}-\partial_{t}v_{\varepsilon})(0)=0,
                \end{array}
              \right.
\end{equation}
where $\rho_{\varepsilon}$ is $C^{\infty}([0,T];H_{(s)}^{-\infty})$-negligible. We now write this as in the following first order system
$$\partial_{t}\begin{pmatrix} W_{1,\varepsilon} \\ W_{2,\varepsilon} \end{pmatrix}
=\begin{pmatrix} 0 & i\R^{1/2}\\ ia_{\varepsilon}(t)\R^{1/2} & 0 \end{pmatrix}
\begin{pmatrix} W_{1,\varepsilon} \\ W_{2,\varepsilon} \end{pmatrix}+\begin{pmatrix} 0 \\ \rho_{\varepsilon} \end{pmatrix},
$$
where
$$W_{1,\varepsilon}=i\R^{1/2}(u_{\varepsilon}-v_{\varepsilon})\;\;\text{and}\;\;W_{2,\varepsilon}=\partial_{t}(u_{\varepsilon}-v_{\varepsilon}).$$
This system will be studied after the group Fourier transform, as a system of the type
$$\partial_{t}V_{\varepsilon}(t,\pi)_{m,k}=i\pi_{m}A_{\varepsilon}(t,\pi)_{m,k}V_{\varepsilon}(t,\pi)_{m,k}+P_{\varepsilon}(t,\pi)_{m,k},$$
for all $\pi \in \widehat{\G}$, and any $m,k\in \mathbb{N}$, where $P_{\varepsilon}(t,\pi)_{m,k}=\begin{pmatrix} 0 \\ \mathcal (F_{\G}\rho_{\varepsilon})_{m,k} \end{pmatrix}$, $A_{\varepsilon}(t,\pi)_{m,k}=\begin{pmatrix} 0 & 1\\ a_{\varepsilon}(t) & 0 \end{pmatrix},$ and $V_{\varepsilon}(0,\pi)_{m,k}=\begin{pmatrix} 0 \\ 0 \end{pmatrix}$.

In this case, we define the energy as
$$E_{\varepsilon}(t,\pi,\delta)_{m,k}:=
(Q_{\varepsilon}(t,\delta)_{m,k}V_{\varepsilon}(t,\pi)_{m,k},V_{\varepsilon}(t,\pi)_{m,k}),$$
where
$$Q_{\varepsilon}(t,\pi)_{m,k}:=\begin{pmatrix} a_{\varepsilon}(t) & 0\\ 0 & 1 \end{pmatrix}+
\delta^{2}\begin{pmatrix} 1 & 0\\ 0 & 0 \end{pmatrix},$$
is the quasi-symmetriser. Then, we have
\begin{multline}\label{last_ineq_Et}
\partial_{t}E_{\varepsilon}(t,\pi,\delta)_{m,k}\\
=(\partial_{t}Q_{\varepsilon}(t,\delta)_{m,k}V_{\varepsilon}(t,\pi)_{m,k},V_{\varepsilon}(t,\pi)_{m,k})
+i\pi_{m}((Q_{\varepsilon}A-A^{*}Q_{\varepsilon})(t)V,V)\\
+2{\rm Re}(Q_{\varepsilon}(t,\delta)_{m,k}P_{\varepsilon}(t,\pi)_{m,k},V_{\varepsilon}(t,\pi)_{m,k}).
\end{multline}
Taking into account the properties in the proof of {\cite[Case 3 of Theorem 1.1]{RT17a}} and continuing
to discuss as in the first part, from \eqref{last_ineq_Et} we conclude that the Cauchy
problem \eqref{weak_prob_state} has a unique solution $u\in \mathcal{G}([0,T];H_{(s)}^{-\infty})$ for all $s\in\mathbb{R}$.

This completes the proof of Theorem \ref{uniq_exist}.
\end{proof}
Now we prove Theorem \ref{consist_weak}
\begin{proof}[Proof of Theorem \ref{consist_weak}] (i) Here, we compare the classical solution $\widetilde{u}$ given by Theorem \ref{Ruzh_Tar_thm} (i) with the very weak solution $u$ provided by Theorem \ref{consist_weak}. By the definition of the classical solution we have for the classical solution $\widetilde{u}$ that
\begin{equation}\label{for_weak_eq_consist1}
\left\{
                \begin{array}{ll}
                  \partial_{t}^{2}\widetilde{u}(t)+a(t)\R \widetilde{u}(t)=0, \;t\in [0,T],\\
                  \widetilde{u}(0)=u_{0}\in L^{2}(\G),\\
                  \partial_{t}\widetilde{u}(0)=u_{1}\in L^{2}(\G).
                \end{array}
              \right.
\end{equation}
From the definition of the very weak solution $u$, we also know that there exists a representative $(u_{\varepsilon})_{\varepsilon}$ of $u$ such that
\begin{equation}\label{for_weak_eq_consist2}
\left\{
                \begin{array}{ll}
                  \partial_{t}^{2}u_{\varepsilon}(t)+a_{\varepsilon}(t)\R u_{\varepsilon}(t)=0, \;t\in [0,T],\\
                  u_{\varepsilon}(0)=u_{0}\in L^{2}(\G),\\
                  \partial_{t}u_{\varepsilon}(0)=u_{1}\in L^{2}(\G),
                \end{array}
              \right.
\end{equation}
for suitable embeddings of $a(t)$. Since $(a_{\varepsilon}-a)_{\varepsilon}\rightarrow0$ in $C([0,T])$ for $a\in L^{\infty}_{1}([0,T])$, then \eqref{for_weak_eq_consist1} becomes
\begin{equation}\label{for_weak_eq_consist3}
\left\{
                \begin{array}{ll}
                  \partial_{t}^{2}\widetilde{u}(t)+a_{\varepsilon}(t)\R \widetilde{u}(t)=n_{\varepsilon}(t), \;t\in [0,T],\\
                  \widetilde{u}(0)=u_{0}\in L^{2}(\G),\\
                  \partial_{t}\widetilde{u}(0)=u_{1}\in L^{2}(\G),
                \end{array}
              \right.
\end{equation}
where $n_{\varepsilon}(t)=(a_{\varepsilon}(t)-a(t))\R \widetilde{u}(t)\in C([0,T];H_{\R}^{s})$ and converges to $0$ in this space as $\varepsilon\rightarrow0$. By virtue of \eqref{for_weak_eq_consist2} and \eqref{for_weak_eq_consist3} we note that $\widetilde{u}-u_{\varepsilon}$ solves the following Cauchy problem
\begin{equation}\label{for_weak_eq_consist4}
\left\{
                \begin{array}{ll}
                  \partial_{t}^{2}(\widetilde{u}-u_{\varepsilon})(t)+a_{\varepsilon}(t)\R (\widetilde{u}-u_{\varepsilon})(t)=n_{\varepsilon}(t), \;t\in [0,T],\\
                  (\widetilde{u}-u_{\varepsilon})(0)=0,\\
                  (\partial_{t}\widetilde{u}-\partial_{t}u_{\varepsilon})(0)=0.
                \end{array}
              \right.
\end{equation}
Then, similarly as in the proof of Theorem \ref{uniq_exist}, we reduce above to a system and apply the group Fourier transform to get the following energy estimate
$$\partial_{t}E_{\varepsilon}(t,\pi)_{m,k}\leq |\partial_{t}a_{\varepsilon}(t)||(\widetilde{V}-V_{\varepsilon})(t,\pi)_{m,k}|^{2}+2|a_{\varepsilon}(t)|
|n_{\varepsilon}(t,\pi)_{m,k}||(\widetilde{V}-V_{\varepsilon})(t,\pi)_{m,k}|$$
for all $m,k\in \mathbb{N}$, which implies
$$\partial_{t}E_{\varepsilon}(t,\pi)_{m,k}\leq c_{1}|(\widetilde{V}-V_{\varepsilon})(t,\pi)_{m,k}|^{2}+c_{2}
|n_{\varepsilon}(t,\pi)_{m,k}||(\widetilde{V}-V_{\varepsilon})(t,\pi)_{m,k}|,$$
since the coefficient $a_{\varepsilon}(t)$ is regular enough. Then, noting $|(\widetilde{V}-V_{\varepsilon})(0,\pi)_{m,k}|=0$ and $n_{\varepsilon}\rightarrow 0$ in $C([0,T];H_{\R}^{s})$ and continuing to discussing as in Theorem \ref{uniq_exist} we arrive at
$$|(\widetilde{V}-V_{\varepsilon})(t,\pi)_{m,k}|\leq c(\omega(\varepsilon))^{\alpha}$$
for some positive constants $c$ and $\alpha$, which concludes that $u_{\varepsilon}\rightarrow \widetilde{u}$ in $C([0,T];H_{\R}^{s+\frac{\nu}{2}})\cap C^{1}([0,T];H_{\R}^{s})$. Furthermore, since any other representative of $u$ will differ from $(u_{\varepsilon})_{\varepsilon}$ by a $C^{\infty}([0,T];H_{\R}^{s})$ - negligible net, the limit is the same for any representative of $u$.

(ii) The Part (ii) can be proven as Part (i) with slight modifications.

This completes the proof of Theorem \ref{consist_weak}.
\end{proof}
\section{Appendix: Inhomogeneous equation case}
\label{SEC:appendix}
In this section we are going to give brief ideas for how to deal with the inhomogeneous
wave equation
\begin{equation}\label{not_weak_thm_eq1}
\left\{
                \begin{array}{ll}
                  \partial_{t}^{2}u(t)+a(t)\R u(t)=f(t), \;t\in [0,T],\\
                  u(0)=u_{0}\in L^{2}(\G),\\
                  \partial_{t}u(0)=u_{1}\in L^{2}(\G).
                \end{array}
              \right.
\end{equation}
\begin{proof}[Proof of Theorem \ref{Ruzh_Tar_thm_inhom} (i)] Let us take the group Fourier
transform of \eqref{not_weak_thm_eq1} with respect to $x\in\G$ for all $\pi\in\widehat{G}$, that is,
\begin{equation}\label{fou_cauchy1}
\partial_{t}^{2}\widehat {u}(t,\pi)+a(t)\pi(\R)\widehat{u}(t,\pi)=\widehat{f}(t,\pi).
\end{equation}
Taking into account \eqref{pi_R_matrix}, we rewrite the matrix equation \eqref{fou_cauchy1} componentwise as an infinite system of equations of the form
\begin{equation}\label{app_eq1}
\partial_{t}^{2}\widehat {u}(t,\pi)_{m,k}+a(t)\pi_{m}^{2}\widehat{u}(t,\pi)_{m,k}=\widehat{f}(t,\pi)_{m,k},
\end{equation}
for all $\pi \in \widehat{\G}$, and any $m,k\in \mathbb{N}$. Now let us decouple the system given by the matrix equation \eqref{fou_cauchy1}. For this, we fix an arbitrary representation $\pi$, and a general entry $(m,k)$ and we treat each equation given by \eqref{app_eq1} individually. If we denote
$$v(t):=\widehat {u}(t,\pi)_{m,k},\;\beta^{2}:=\pi_{m}^{2},\;f(t):=\widehat{f}(t,\pi)_{m,k}$$
and
$$v_{0}:=\widehat {u}_{0}(\pi)_{m,k},\;v_{1}:=\widehat {u}_{1}(\pi)_{m,k},$$
then \eqref{app_eq1} becomes
\begin{equation}\label{second_order}
v''(t)+\beta^{2}a(t)v(t)=f(t).
\end{equation}
Denoting
\begin{equation*}V(t):=\left(
          \begin{array}{cc}
            i\beta v(t) \\
            \partial_{t}v(t)
          \end{array}
        \right), V_{0}:=\left(
          \begin{array}{cc}
            i\beta v_{0} \\
            v_{1},
          \end{array}
        \right),
        F(t):=\left(
          \begin{array}{cc}
            0 \\
            f(t)
          \end{array}
        \right),
\end{equation*}
and the matrix
\begin{equation*}A(t):=\left(
                   \begin{array}{cc}
                     0 & 1 \\
                     a(t) & 0 \\
                   \end{array}
                 \right),
\end{equation*}
we reduce the second order system \eqref{second_order} to the following first order system
\begin{equation}\label{not_weak_prop_1ord}
\left\{
                \begin{array}{ll}
                  V'(t)=i\beta A(t)V(t)+F(t), \;t\in [0,T],\\
                  V(0)=V_{0}.
                \end{array}
              \right.
\end{equation}
Note that the eigenvalues of matrix $A(t)$ are $\pm \sqrt{a(t)}$. Let $S$ be the symmetriser of $A$, that is, the matrix $S$ satisfies
 $$SA-A^{*}S=0,$$
has the form
\begin{equation*}S(t):=\left(
                   \begin{array}{cc}
                     2a(t) & 0 \\
                     0 & 2 \\
                   \end{array}
                 \right).
\end{equation*}
Now let us define the energy as
$$E(t):=(S(t)V(t),V(t)).$$
By a direct calculation we have
\begin{equation}\label{two_side_est_E}
2|V(t)|^{2}\min_{t\in [0,T]}\{a(t),1\}\leq E(t)\leq 2|V(t)|^{2}\max_{t\in [0,T]}\{a(t),1\}.
\end{equation}
Since $a(t)>0$ is continuous, there exist positive constants $a_{0},a_{1}>0$ such that
$$a_{0}=\min_{t\in [0,T]}a(t)\;\;\text{and}\;\;a_{1}=\max_{t\in [0,T]}a(t).$$
Denoting $c_{0}:= 2\min\{a_{0}, 1\}$ and $c_{1}:=2\max\{a_{1}, 1\}$, we rewrite \eqref{two_side_est_E} as
\begin{equation}\label{two_side_est_E_2}
c_{0}|V(t)|^{2}\leq E(t)\leq c_{1}|V(t)|^{2}.
\end{equation}
A direct calculation with \eqref{two_side_est_E_2} implies that
\begin{equation}\label{estimating_E(t)}
\begin{split}
E_{t}&(t)\\&=(S_{t}(t)V(t),V(t))+(S(t)V_{t}(t),V(t))+(S(t)V(t),V_{t}(t))\\&
=(S_{t}(t)V(t),V(t))+i\beta(S(t)A(t)V(t),V(t))+(S(t)F(t),V(t))\\&-i\beta (S(t)V(t),A(t)V(t))+(S(t)V(t),F(t))\\&
=(S_{t}(t)V(t),V(t))+i\beta((S(t)A(t)-A^{*}(t)S(t))V(t),V(t))+2{\rm Re} (S(t)F(t),V(t))\\&
=(S_{t}(t)V(t),V(t))+2{\rm Re} (S(t)F(t),V(t))\\&
\leq (\|S_{t}\|+1)|V(t)|^{2}+\|SF\|^{2}\\&
\leq \max\{\|S_{t}\|+1, \|S\|^{2}\}(|V(t)|^{2}+|F(t)|^{2})\\&
\leq C_{1} E(t)+C_{2}|F(t)|^{2}
\end{split}
\end{equation}
for some positive constants $C_{1}$ and $C_{2}$. Applying Gronwall's lemma and noting \eqref{two_side_est_E_2}, we get
\begin{equation}\label{V_estim}
|V(t)|^{2}\leq c_{0}^{-1}E(t)\leq C_{1}|V_{0}|^{2}+C_{2}\sup_{0\leq t\leq T}|F(t)|^{2},
\end{equation}
for all $t\in[0,T]$ with \enquote{new} constants $C_{1}$ and $C_{2}$ depending on $T$. Recalling the definition of $V(t)$, the last inequality gives
$$\beta^{2}|v(t)|^{2}+|v'(t)|^{2}\leq C (\beta^{2}|v_{0}|^{2}+|v_{1}|^{2}+\sup_{0\leq t\leq T}|f(t)|^{2}),$$
which is equivalent to
\begin{multline}\label{app_eq2}
|\pi_{m}\widehat{u}(t,\pi)_{m,k}|^{2}+|\partial_{t}\widehat{u}(t,\pi)_{m,k}|^{2}\\ \leq
C(|\pi_{m}\widehat {u}_{0}(\pi)_{m,k}|^{2}+|\widehat {u}_{1}(\pi)_{m,k}|^{2}+\sup_{0\leq t\leq T}|\widehat{f}(t,\pi)_{m,k}|^{2}).
\end{multline}
This holds uniformly in $\pi\in\widehat{\G}$ and $m,k\in\mathbb{N}$. We multiply the inequality \eqref{app_eq2} by $\pi_{m}^{4s/\nu}$ to get
\begin{multline}\label{app_eq3}
|\pi_{m}^{1+\frac{2s}{\nu}}\widehat{u}(t,\pi)_{m,k}|^{2}+|\pi_{m}^{\frac{2s}{\nu}}\partial_{t}\widehat{u}(t,\pi)_{m,k}|^{2}\\ \leq
C(|\pi_{m}^{1+\frac{2s}{\nu}}\widehat {u}_{0}(\pi)_{m,k}|^{2}
+|\pi_{m}^{\frac{2s}{\nu}}\widehat {u}_{1}(\pi)_{m,k}|^{2}+\sup_{0\leq t\leq T}|\pi_{m}^{\frac{2s}{\nu}}\widehat{f}(t,\pi)_{m,k}|^{2}).
\end{multline}
Thus, since for any Hilbert-Schmidt operator $A$ one has
$$\|A\|^{2}_{{\rm HS}}=\sum_{m,k}|(A\phi_{m},\phi_{k})|^{2}$$
for any orthonormal basis $\{\phi_{1},\phi_{2},\ldots\}$, then we can consider the infinite sum over $m,k$ of the inequalities provided by \eqref{app_eq3}, to obtain
\begin{multline}\label{app_eq4}
\|\pi(\R)^{\frac{1}{2}+\frac{s}{\nu}}\widehat{u}(t,\pi)\|^{2}_{{\rm HS}}+
\|\pi(\R)^{\frac{s}{\nu}}\partial_{t}\widehat{u}(t,\pi)\|^{2}_{{\rm HS}}\\ \leq
C(\|\pi(\R)^{\frac{1}{2}+\frac{s}{\nu}}\widehat{u}_{0}(\pi)\|^{2}_{{\rm HS}}+
\|\pi(\R)^{\frac{s}{\nu}}\widehat{u}_{1}(\pi)\|^{2}_{{\rm HS}}+
\|\pi(\R)^{\frac{s}{\nu}}\widehat{f}(t,\pi)\|^{2}_{C([0,T];{\rm HS})}).
\end{multline}
Thus, integrating both sides of \eqref{app_eq4} against the Plancherel measure $\mu$ on $\widehat{\G}$, then using the Plancherel identity \eqref{planch_for} we obtain \eqref{est_tar1_inhom}.
\end{proof}
Similarly, Theorem \ref{Ruzh_Tar_thm} (ii), (iii) and (iv) imply Parts (a) of Theorem \ref{Ruzh_Tar_thm_inhom} (ii), (iii) and (iv), respectively. Then, Parts (b) of Theorem \ref{Ruzh_Tar_thm_inhom} (ii), (iii) and (iv) can be proved as in homogeneous cases using Parts (a) of Theorem \ref{Ruzh_Tar_thm_inhom} (ii), (iii) and (iv). In the same way as in the proof of homogeneous cases of Theorems \ref{exist_weak}, \ref{uniq_exist}, \ref{consist_weak} and \ref{consist2_weak}, their inhomogeneous cases can be proven with slight modifications.

\end{document}